\documentclass[a4,11pt]{amsart}%
\usepackage{amsmath,amssymb,amsfonts,enumerate,amsthm, amscd,}
\usepackage[all]{xy}
\usepackage{graphicx}
\usepackage{amsmath}
\usepackage{amsfonts}
\usepackage{amssymb}
\usepackage{hyperref}

\providecommand{\U}[1]{\protect\rule{.1in}{.1in}}

\newcommand{\N}{\mathbb{N}}
\newcommand{\lo}{\longrightarrow}

\newtheorem{thm}{Theorem}[section]
\newtheorem{cor}[thm]{Corollary}
\newtheorem{lem}[thm]{Lemma}
\newtheorem{prop}[thm]{Proposition}
\newtheorem{defn}[thm]{Definition}

\newtheorem{exam}[thm]{Example}

\def\proof{{\parindent0pt {\bf Proof.\ }}}

\newcommand{\field}[1]{\mathbb{#1}}

\newcommand{\Z }{\field{Z}}
\theoremstyle{definition}

\theoremstyle{remark}

\theoremstyle{Definition and Notation}

\begin{document}

\title[When every finitely generated ideal is $S$-principal]{When every finitely generated ideal is $S$-principal}

\author[M. Chhiti]{Mohamed Chhiti}
\address{Mohamed Chhiti\\Laboratory of Modelling and Mathematical Structures, \\ Faculty of Economics and Social Sciences of Fez,
	University S.M. Ben Abdellah Fez, Morocco.
	$$E-mail\ address:\ chhiti.med@hotmail.com$$}

\author[S. E. Mahdou]{Salah Eddine Mahdou}
\address{Salah Eddine Mahdou \\Laboratory of Modelling and Mathematical Structures, \\ Faculty of Science and Technology of Fez, Box 2202, University S. M.
Ben Abdellah Fez, Morocco.
 $$E-mail\ address:\ salahmahdoulmtiri@gmail.com$$}

 \author[M. A. S. Moutui]{Moutu Abdou Salam Moutui}

 \address{Moutu Abdou Salam Moutui\\Division of Science, Technology, and Mathematics\\ American University of Afghanistan, Doha, Qatar.}
\email{mmoutui@auaf.edu.af}



\begin{abstract}
In this paper, we introduce the concept of $S$-B\'ezout ring, as a generalization of B\'ezout ring. We investigate the relationships between $S$-B\'ezout and other related classes of rings. We establish some characterizations of $S$-B\'ezout rings. We study this property in various contexts of commutative rings including direct product, localization, trivial ring extensions and amalgamation rings. Our results allow us to construct new original classes of $S$-B\'ezout rings subject to various ring theoretical properties. Furthermore, we introduce the notion of nonnil $S$-B\'ezout ring and establish some characterizations.
\end{abstract}

\maketitle


\bigskip

\section{Introduction}
\bigskip
Throughout this paper, all rings are assumed to be commutative with nonzero identity and all modules are nonzero unital. In the last forty years, there have been many extensions of classical notions used in multiplicative ideal theory, with the need of generalizing many properties, including B\'ezout domains, to broader contexts. The notion of B\'ezout ring is central in commutative algebra and enlarging this class as well as producing new original examples can be of some interest. In \cite{AZf}, Anderson and Zafrullah extended the class of B\'ezout domains in the following way: they called a domain $A$ an almost B\'ezout domain if, given any two elements $a,b\in A$, there exists an integer $n\geq 1$ such that the ideal $(a^n,b^n)$ of $A$ is principal. Later, Mahdou, Mimouni and Moutui enlarged the notion almost B\'ezout domain to ring with zero divisors. In recent years, the concept of $S$-property has an important place in commutative algebra and they draw attention by several authors. In \cite{AD}, Anderson and Dumitrescu introduced the concept of $S$-finite modules, where $S$ is a multiplicatively subset as follows: an $R$-module $M$ is called an $S$-finite module if there exist a finitely generated $R$-submodule $N$ of $M$ and $s\in S$  such that $sM\subseteq N.$  Also, they introduced the concept of $S$-Noetherian rings as follows : a ring $R$ is called $S$-Noetherian if every ideal of $R$ is $S$-finite. Recently, in \cite{BE}, Bennis and El Hajoui investigated the  $S$-versions of finitely presented modules and coherent modules which are called, respectively, $S$-finitely presented modules and S-coherent modules. An $R$-module $M$ is called an $S$-finitely presented module for some multiplicative closed subset $S$ of $R$  if there exists  an exact sequence of $R$-modules $0 \to K \to F \to M \to 0$, where $F$ is a finitely generated free $R$-module and $K$ is an $S$-finite $R$-module. Moreover,  an $R$-module $M$ is said to be $S$-coherent, if it is finitely generated and every finitely generated submodule of $M$ is $S$-finitely presented. They showed that the $S$-coherent rings have a similar characterization to the classical one given by Chase for coherent rings (see \cite[Theorem 3.8]{BE}$)$. Any coherent ring is $S$-coherent and any $S$-Noetherian ring is $S$-coherent. For more detail about $S$-Noetherian and $S$-coherent rings, we refer the reader to \cite{AD, BE}. \\

Let $R$ be a ring and $M$ be an $R$-module. Then $R\propto M$, the trivial (ring) extension of $R$ by $M$, is
the ring whose additive structure is that of the external direct sum $R\oplus
M$ and whose multiplication is defined by $(r_{1},m_{1})(r_{2},m_{2}):=(r_{1}r_{2},r_{1}m_{2}+r_{2}m_{1})$ for all $r_{1},r_{2}\in R$ and all $m_{1},m_{2}\in M$. This construction is also known by other terminology,
such as the idealization. See for instance \cite{G,H}. Trivial ring extensions have
an important place in commutative ring theory, due to their effectiveness of producing new classes of examples and counter examples of rings subject to various
ring theoretic properties (see for example \cite{MMA, AW, G, H, KM}). \\

Let $(A,B)$ be a pair of rings, $J$ be an ideal of $B$ and $f: A\rightarrow B$ be a ring homomorphism. The amalgamation of $A$ with $B$ along $J$ with respect to $f$ is the following subring of $A\times B$:
 $$A\bowtie^{f}J = \{(a,f(a)+j) \diagup a \in A, j \in J\}.$$
  As a natural generalization of the
duplication construction in \cite{DF}, the amalgamation ring was initiated by D'Anna, Finocchiaro, and Fontana in
\cite{DFF}. This ring extension covers several ring theoretic constructions, including
the $A+XB[X]$, $A+XB[[X]]$, and the $D+M$ constructions as well as several classical constructions, such as trivial ring extensions and the CPI extensions. For more details regarding the amalgamation ring, we refer the reader to \cite{DFF, D, DF, EKM}.\\

It is worthwhile recalling that a prime ideal $P$ of $R$ is {\it divided} if $P \subseteq (x)$ for any $x \in R \setminus P$. A ring $R$ is said to be {\it divided} if every prime ideal of $R$ is divided. For more details on divided rings and related notions, we refer the reader to \cite{ayb1}. Badawi has extensively investigated in \cite{ayb4, ayb2, ayb3, ayb5, ayb6}, the following class of rings:
$\mathcal{H} = \{R \mid R$ is a commutative ring and $Nil(R)$ is a divided prime ideal of $R\}$.
Note that in the case $R \in \mathcal{H}$, $R$ is called a {\it $\phi$-ring}. Obviously, every integral domain is a $\phi$-ring. An ideal $I$ of $R$ is said to be a {\it nonnil ideal} if $I \nsubseteq$ $Nil(R)$. If $I$ is a nonnil ideal of a $\phi$-ring $R$, then $Nil(R)$ $\subseteq I$. Let $R$ be a ring with total quotient ring $T$ such that $Nil(R)$ is a divided prime ideal of $R$. As in \cite{ayb2}, we define $\phi : T(R) \longrightarrow K := R_{Nil(R)}$ such that $\phi(\frac{a}{b}) = \frac{a}{b}$ for every $a \in R$ and every $b \in R\setminus Z(R)$. Then $\phi$ is a ring homomorphism from $T(R)$ into $K$, and $\phi$ restricted to $R$ is also a ring homomorphism from $R$ into $K$ given by $\phi(x) = \frac{x}{1}$ for every $x \in R$. Observe that if $R \in \mathcal{H}$, then $\phi(R) \in \mathcal{H}$, $Ker(\phi) \subseteq Nil(R)$, $Nil(T(R)) = Nil(R)$, $Nil(R_{Nil(R)}) = \phi(Nil(R)) = Nil(\phi(R)) = Z(\phi(R))$, $T(\phi(R)) = R_{Nil(R)}$ is quasilocal with maximal ideal $Nil(\phi(R))$, and $R_{Nil(R)}/Nil(\phi(R)) = T(\phi(R))/Nil(\phi(R))$ is the quotient field of $\phi(R)/Nil(\phi(R))$.\\

In this paper, we introduce the concept of $S$-B\'ezout ring, as a generalization of B\'ezout ring. We investigate the relationships between $S$-B\'ezout and other related classes of rings. We establish some characterizations of $S$-B\'ezout rings. We study this property in various contexts of commutative rings including direct product, localization, trivial ring extensions and amalgamation rings. Our results allow us to construct new original classes of $S$-B\'ezout rings subject to various ring theoretical properties. Moreover, we introduce the notion of nonnil $S$-B\'ezout ring and establish some characterizations. For a ring $R$, we denote respectively by $Nil(R)$, $U(R)$, $Z(R)$ the ideal of all nilpotent elements of $R$, the multiplicative group of units of $R$ and the zero divisor elements of $R$. For an integral domain $R$, we denote by $qf(R)$, the quotient field of $R$.

\section{Main Results}
Let $R$ be a ring and let $S$ be a multiplicatively closed subset of $R$. Recall that a ring $R$ is B\'ezout if every finitely generated ideal of $R$ is principal. First, we introduce a weak version of B\'ezout ring called $S$-B\'ezout ring in the following way:
\begin{defn}
Let $R$ be ring and $S$ be a multiplicative set of $R$. An ideal $I$ of $R$ is said to be $S$-principal if there exists $s\in S$ such that $sI\subseteq rR\subseteq I$ for some $r\in R$. $R$ is said to be $S$-B\'ezout, if every finitely generated ideal of $R$ is $S$-principal.
\end{defn}

It is worthwhile noting that any B\'ezout ring is $S$-B\'ezout for every multiplicative set $S$ of $R$. However, the converse is not true in general as shown in the following examples which illustrate non-B\'ezout $S$-B\'ezout rings. \\
\begin{exam}\label{ex}
Let $R:=\mathbb{Z}\propto (\mathbb{Z}/{2\mathbb{Z}})^\infty$ be the trivial ring extension of $\mathbb{Z}$ by the $\mathbb{Z}$-module $(\mathbb{Z}/{2\mathbb{Z}})^\infty$ and $S :=\{2^n ,0) / n \in \mathbb{N}\}$ be a multiplicatively set of $R$. Then:
\begin{enumerate}
\item $R$ is an $S$-B\'ezout ring.
\item $R$ is not a B\'ezout ring.
\end{enumerate}
\end{exam}

\begin{proof}
\begin{enumerate}
\item Let $J$ be a finitely generated proper ideal of $R$ and set $I :=\{a \in \mathbb{Z} / (a,e) \in J$ for some $e \in (\mathbb{Z}/{2\mathbb{Z}})^\infty \}$. Two cases are then possible: \\
    Case 1. $I \neq 0$. \\
    In this case, say $I =n\mathbb{Z}$, for some $n \in \mathbb{Z}-\{0\}$. Let $e \in (\mathbb{Z}/{2\mathbb{Z}})^\infty$ such that $(n,e) \in J$. So, $J$ is $S$-principal since
    $$(2,0)J \subseteq R(n,e) \subseteq J $$
    since $(2,0)J =R(2n,0) =R(n,e)(2,0) \subseteq R(n,e)$, as desired. \\
    Case 2. $I = 0$. \\
    Then $J$ is also $S$-principal since
    $$(2,0)J \subseteq R(0,0) \subseteq J $$
    since $(2,0)J \subseteq (2,0)(0\propto (\mathbb{Z}/{2\mathbb{Z}})^\infty ) =R(0,0)$, as desired.
\item Let $(e,f)$ be linearly independent elements of the $\mathbb{Z}$-module $(\mathbb{Z}/{2\mathbb{Z}})^\infty $ and set
$J =R(0,e) + R(0,f)$. It is clear that $J$ is not a principal ideal of $R$, as desired.  \qed\\
\end{enumerate}
\end{proof}

\begin{exam}\label{thm3.10}
 Let $R_1$ be a B\'ezout ring, $R_2$ be a non-B\'ezout ring, $R :=R_1 \times R_2$ and set $S :=\{(1,1),(1,0)\}$ be a multiplicative set of $R$. For instance take $R_1:=K[x]$ the polynomial ring with coefficient in a field $K$, and $R_2:=F[X,Y]$, the polynomial ring in two variables $X,Y$ over a field $F$. It is well known that $R_2$ is a non-B\'ezout domain. Then:
 \begin{enumerate}
 \item $R$ is an $S$-B\'ezout ring.
 \item $R$ is a non-B\'ezout ring.
 \end{enumerate}
\end{exam}
\begin{proof}
\begin{enumerate}
 \item We claim that $R$ is an $S$-B\'ezout ring. Indeed, let $I :=I_1 \times I_2$ be a finitely generated ideal of $R :=R_1 \times R_2$. Hence,
 $I_1$ is a finitely generated ideal of a B\'ezout ring $R_1$ and so $I_1 =R_1 a_1$ for some $a_1 \in R_1$. Hence,
 $$(1,0)(I_1 \times I_2) \subseteq I_1 \times 0 (=R(a_1 ,0)) \subseteq I_1 \times I_2$$
 and so $I :=I_1 \times I_2$ is an $S$-principal ideal, as desired.
 \item $R :=R_1 \times R_2$ is a non-B\'ezout ring since $R_2$ is a non-B\'ezout ring.\qed\\

 \end{enumerate}
\end{proof}

Recall that a ring $R$ is called an $S$-PIR if every ideal is $S$-principal, that is for every ideal $I$ there exist a principal ideal $J$ and $s\in S$  such that $sI\subseteq J \subseteq I$. Clearly, any $S$-PIR is an $S$-B\'ezout ring for any multiplicatively closed subset $S$ of $R$.\\

\begin{exam}\label{0}
  Let $R$ be any non-B\'ezout domain  and set $S :=R - \{0\}$ be a multiplicatively set of $R$. Then:
   \begin{enumerate}
 \item $R$ is $S$-PIR. In particular, $R$ is $S$-B\'ezout ring.
 \item $R$ is not a B\'ezout ring.
  \end{enumerate}
\end{exam}

\begin{proof}
   \begin{enumerate}
 \item Let $I$ be a proper ideal of $R$ and let $s \in I - \{0\}$. Hence, $sI \subseteq Rs \subseteq I$ and so $I$ is $S$-principal since $Rs$ is a principal ideal of $R$, as desired. \\
\item By hypothesis. \qed\\
\end{enumerate}
\end{proof}

Recall that an ideal $I$ is said $n$-generated if it can be generated by $n$ elements. It is well known that a B\'ezout ring is a ring in which every $n$-generated (resp., $2$-generated ideal) is principal. Now, we generalize this result to $S$-B\'ezout ring. \\

\begin{prop}\label{prop3.4}
Let $R$ be ring and $S$ be a multiplicative set of $R$. Then $R$ is an $S$-B\'ezout ring if and only if every $2$-generated ideal is $S$-principal.
\end{prop}

\begin{proof}
If $R$ is an $S$-B\'ezout ring, then every $2$-generated ideal is $S$-principal. Conversely, assume that every $2$-generated ideal is $S$-principal. It suffices to show that if every $n$-generated ideal is $S$-principal, then every $(n+1)$-generated ideal is $S$-principal.\\
 Assume that every $n$-generated ideal is $S$-principal and let $I =\sum_{i=1}^{n+1}Ra_i$ be an $(n+1)$-generated ideal of $R$. Then there
 exist $s \in S$ and $a \in R$ such that
 $$s\sum_{i=1}^{n}Ra_i \subseteq Ra \subseteq \sum_{i=1}^{n}Ra_i$$
 by hypothesis since $\sum_{i=1}^{n}Ra_i$ is an $n$-generated ideal of $R$. Using the fact that $sRa_{n+1} \subseteq Ra_{n+1}$, we have
 $$s\sum_{i=1}^{n+1}Ra_i \subseteq s\sum_{i=1}^{n}Ra_i + Ra_{n+1} \subseteq Ra + Ra_{n+1} \subseteq \sum_{i=1}^{n+1}Ra_i.$$
On the other hand, $Ra + Ra_{n+1}$ is $2$-generated, then by hypothesis, there exist $s' \in S$ and $b \in R$ such that
 $$s'(Ra + Ra_{n+1}) \subseteq Rb \subseteq Ra + Ra_{n+1} (\subseteq \sum_{i=1}^{n+1}Ra_i).$$
 But, we have $ss'\sum_{i=1}^{n+1}Ra_i \subseteq s'(Ra + Ra_{n+1})$. Therefore,
 $$ss'\sum_{i=1}^{n+1}Ra_i \subseteq Rb \subseteq \sum_{i=1}^{n+1}Ra_i$$
 as $ss' \in S$, which completes the proof.
   \qed\\
\end{proof}

\bigskip
The next theorem establishes another characterization of $S$-B\'ezout ring.
\begin{thm}\label{thm30}
 Let $R$ be a ring and $S$ be a multiplicative set of $R$. Then $R$ is an $S$-B\'ezout ring if and only if every $S$-finite ideal of $R$ is $S$-principal.
\end{thm}
\begin{proof}
If every $S$-finite ideal of $R$ is $S$-principal, then it is clear that $R$ is $S$-B\'ezout. Conversely, assume that $R$ is $S$-B\'ezout
and let $I$ be an $S$-finite ideal of $R$. Then there exist a finitely generated ideal $J$ of $R$ and $s \in S$ such that:
$$sI \subseteq J \subseteq I.$$
Hence, $J$ is $S$-principal and so there exist $s' \in S$ and $a \in R$ such that
$$s'J \subseteq Ra \subseteq J.$$
Therefore,
$$ss'I  \subseteq s'J \subseteq Ra \subseteq J \subseteq I$$
and so $I$ is $S$-principal, as desired.  \qed\\
\end{proof}

Next, we investigate the stability of $S$-B\'ezout ring property under localization.

\begin{prop}\label{prop3.4}
Let $R$ be a ring and $S\subseteq R$ be a multiplicative set of $R$. Then:
 \begin{enumerate}
 \item If $R$ is an $S$-B\'ezout ring, then $S^{-1}R$ is a B\'ezout ring.
 \item Assume that $R$ is $S$-Noetherian (not necessary Noetherian). Then $R$ is an $S$-B\'ezout ring if and only if $S^{-1}R$ is a B\'ezout ring.
 \end{enumerate}
\end{prop}

\begin{proof}
$(1)$ Assume that $R$ is $S$-B\'ezout and let $J$ be a finitely generated ideal of $S^{-1}R$. There exists a finitely generated ideal $I$ of $R$ such that
$J =S^{-1}I$.  Then $I$ is $S$-principal and hence $J :=S^{-1}I$ is principal, as desired.\\
$(2)$ If $R$ is an $S$-B\'ezout ring, then by assertion $(1)$ above, $S^{-1}R$ is a B\'ezout ring. Conversely, assume that $S^{-1}R$ is a B\'ezout ring with the hypothesis that $R$ is $S$-Noetherian (not necessarily Noetherian). Then $S^{-1}R$ is a PIR. So, $R$ is an $S$-PIR by \cite[Proposition 2(g)]{AD}. Hence, $R$ is an $S$-B\'ezout ring, as desired. \qed
\end{proof}

\bigskip

Contrary to Proposition \ref{prop3.4}, the following proposition shows that we do not always need the assumption "$S$-Noetherian".

\begin{prop}\label{prop3.5}
Let $R$ be a ring and $S\subseteq U(R)$ be a multiplicative set of $R$. Then $R$ is an $S$-B\'ezout ring if and only if $S^{-1}R$ is a B\'ezout ring.
\end{prop}

\bigskip

If we choose $S \subseteq U(R)$, then we obtain the following known corollary.

\begin{cor}\label{cor3.7}
Every Noetherian B\'ezout ring is a PIR.
\end{cor}

\bigskip

For a prime ideal $P$ of $R$, we define $R$ to $P$-B\'ezout if $R$ is $(R-P)$-B\'ezout. We establish the following characterization for a semilocal ring.

\begin{thm}\label{thm3.10}
For a semilocal ring $R$, the following statements are equivalent:
\begin{enumerate}
\item $R$ is a B\'ezout ring.
\item $R$ is a $P$-B\'ezout ring for all prime ideal $P$ of $R$.
\item $R$ is an $M$-B\'ezout ring for all maximal ideal $M$ of $R$.
\end{enumerate}
\end{thm}

\begin{proof}
$(1)\Rightarrow(2)$ Straightforward.\\
$(2)\Rightarrow (3)$ Clear.\\
$(3)\Rightarrow (4)$ We can use the same proof as in \cite[Proposition 12]{AD}. \qed
\end{proof}

\bigskip
The next proposition examines the $S$-B\'ezout ring property under homomorphic image.

\begin{prop}\label{prop2.4}
Let $R$ be a ring, $I$ be a finitely generated ideal of $R$ and $S$ be a multiplicatively set of $R$. If $R$ is an $S$-B\'ezout ring, then $R/I$ is an $(S+I)$-B\'ezout ring. The converse is true if there is $s_0\in S$ such that $s_0I=0$.
\end{prop}

\begin{proof}
Assume that $R$ is $S$-B\'ezout. Let $J/I$ be a finitely generated ideal of $R/I$, where $J$ is a finitely generated ideal of $R$. Then there exist
$s \in S$ and $a \in R$ such that $$sJ \subseteq Ra \subseteq J$$ since $R$ is a B\'ezout ring. Hence, we have
$$(s+I)(J/I) \subseteq (R/I)(a+I) \subseteq J/I$$ and so $J/I$ is $(S+I)$-principal. Hence, $R/I$ is $(S+I)$-B\'ezout. Conversely, assume that $R/I$ is $(S+I)$-B\'ezout and there exists $s_0\in S$ such that $s_0 I=0$. Consider a finitely generated ideal $J$ of $R$. Then $((J+I)/I)$ is a finitely generated ideal of an $(S+I)$-B\'ezout ring $R/I$. Hence, there exist $s \in S$ and $a \in R$ such that:

$$(s+I)((J+I)/I)\subseteq (R/I)(a+I)\subseteq (J+I)/I.$$

By multiplying the above equation by $s_0$ and since $s_0 I =0$, we obtain :

$$ss_0 J\subseteq R(s_0a)\subseteq s_0J \subseteq J,$$
and so $J$ is $S$-principal. Hence, $R$ is an $S$-B\'ezout ring which completes the proof. \qed

\end{proof}

\bigskip
Now we study the stability of $S$-B\'ezout property under homomorphism. Let $A$ and $B$ be two rings, $f:A\rightarrow B$ be a ring
homomorphism and $I$ be an ideal of $A$. We denote by $I^{e}$, the
extension of $I$ is defined by the ideal of $B$ generated by $f(I)$;
and by $J^{c}$ the contraction of $J$ defined by the set of
antecedents of elements of $f(A)\cap J$, that is, $J^{c}=\{a\in
A|f(a)\in J\}$.

\begin{prop}
Let $(A,B)$ be a pair of rings, $f:A\rightarrow B$ be a ring
homomorphism and $S$ be a multiplicative set of $A$ such that
$I^{ce}=I$ for each ideal $I$ of $B$ and $J^{c}$ is a finitely generated
ideal of $A$ for each finitely generated ideal $J$ of $B$. If $A$ is an
$S$-B\'ezout ring, then $B$ is an $f(S)$-B\'ezout ring.
\end{prop}
\begin{proof}
Assume that $A$ is an $S$-B\'ezout ring. Let $J$ be a finitely generated ideal of $B$. From assumption $J^c$ is a finitely generated ideal of $A$ which is $S$-B\'ezout. So, $J^c$ is $S$-principal and therefore there exist $s\in S$ and $a\in A$ such that $sJ^c\subseteq aA\subseteq J^c$. It follows that $f(s)J=f(s)J^{ce}\subseteq (aA)^e\subseteq J^{ce}=J$. Hence, $f(s)J\subseteq f(a)B\subseteq J$, making $J$, an $f(S)$-principal ideal of $B$. Finally, $B$ is an $f(S)$-B\'ezout ring, as desired.\qed\\
\end{proof}

It is clear that if $S_1,...,S_n$ are multiplicative sets of rings $R_1,...,R_n$ respectively, then $S=\prod_{i=1}^{n}S_{i}$ is a multiplicative set of $R=\prod_{i=1}^{n}R_{i}$. Next, we study the stability of the $S$-Bezout property under direct product. \\

\begin{prop}\label{prop2.6}
Let $S_1,...,S_n$ be multiplicative sets of rings $R_1,....,R_n$ respectively. Set $R=\prod_{i=1}^{n}R_{i}$ and $S=\prod_{i=1}^{n}S_{i}$ a multiplicatively closed subset of $R$. The following statements are equivalent: \begin{enumerate}
\item $R$ is an $S$-B\'ezout ring.
\item $R_i$ is an $S_i$-B\'ezout ring for all $i=1,...,n$.
\end{enumerate}
\end{prop}

\begin{proof}
It suffices to prove the result for $n =2$. \\ Assume that $R_1\times R_2$ is an $(S_1\times S_2)$-B\'ezout ring and let $I_i$ be a finitely generated ideal of $R_i$, for $i =1, 2$. Then $I_1 \times I_2$ is a finitely generated ideal of an $(S_1\times S_2)$-B\'ezout ring $R_1\times R_2$. Hence, there exist $s_i \in S_i$ and $a_i \in R_i$ such that

$$(s_1,s_2)(I_1 \times I_2) \subseteq (R_1\times R_2)(a_1,a_2) \subseteq I_1 \times I_2,$$
and so $s_1 I_1 \subseteq R_1 a_1 \subseteq I_1 $ and $s_2 I_2 \subseteq R_2 a_2 \subseteq I_2 $. Hence, $R_i$ is an $S_i$-B\'ezout ring, as desired.
Conversely, assume that $R_i$ is an $S_i$-B\'ezout ring for each $i=1,2$ and let $I_1\times I_2$ be a finitely generated ideal of $R_1\times R_2$.  Hence, there exist $s_i \in S_i$ and $a_i \in R_i$ such that $s_1 I_1 \subseteq R_1 a_1 \subseteq I_1 $ and $s_2 I_2 \subseteq R_2 a_2 \subseteq I_2 $ and so we have
$$(s_1,s_2)(I_1 \times I_2) \subseteq (R_1\times R_2)(a_1,a_2) \subseteq I_1 \times I_2.$$
Hence, $R_1\times R_2$ is an $(S_1\times S_2)$-B\'ezout ring which completes the proof.  \qed
\end{proof}
\bigskip
The next proposition studies the $S$-B\'ezout property under the ring extension $A\subseteq B$, where $(A,B)$ is a pair of rings.

\begin{prop}
  Let $A\subseteq B$ be a ring extension such that $IB\cap A=I$ for each ideal $I$ of $A$ and $S\subseteq A$ a multiplicative set. If $B$ is an $S$-B\'ezout ring, then so is $A$.
\end{prop}
\begin{proof}
Let $I$ be a finitely generated ideal of $A$. Since $B$ is an $S$-B\'ezout ring, then there exist $s\in S$ and $b\in B$ such that $sIB\subseteq bB\subseteq IB$. Clearly $b\in IB$ and so can be picked from $I$. Therefore, $sI=sIB\cap A\subseteq bB\cap A\subseteq bA\subseteq I$, making $I$ an $S$-principal ideal of $A$. Hence, $A$ is an $S$-B\'ezout ring, as desired.\qed\\
\end{proof}

Let $R:= A \propto E$ be the trivial ring extension of a ring $A$ by an $A$-module $E$. Note that if $S$ is a multiplicative set of $R$, then $S_0 = \{a \in A / (a,e)\in S$ for some $ e\in E\}$ is a multiplicative set of $A$. Conversely, if $S_0$ is a multiplicative set of $A$, then $S :=S_0 \propto N$ is a multiplicative set of $R$ for every submodule $N$ of $E$ such that $S_{0}N \subseteq N$. In particular, $S_{0} \propto 0$ and $S_{0} \propto E$ are multiplicative sets of $R$. Our next result examines the transfer of the $S$-B\'ezout property in trivial ring extension in the special setting $E$ is a finitely generated $A$-module.\\

\begin{thm}\label{thm2.2}
Let $A$ be a ring, $E$ be a finitely generated $A$-module, $S_0$ be a multiplicatively closed subset of $A$, $R:= A\propto E$ be the trivial ring extension of $A$ by $E$ and $S:= S_0\propto E$ be a multiplicatively closed subset of $R$. If $R$ is an $S$-B\'ezout ring, then $A$ is an $S_0$-B\'ezout ring and every finitely generated submodule $F$ of $E$ is $S$-cyclic (in particular $E$ is $S$-cyclic). The converse holds if every ideal of $R$ is homogeneous.
\end{thm}

Before proving the previous theorem, we establish the following lemma.

\begin{lem}\label{lem2.3}
Let $A$ be a ring, $E$ be an $A$-module, $F$ be a submodule of $E$ and $S$ be a multiplicatively closed subset of $A$. Then $F$ is $S$-cyclic if and only if $0\propto F$ is an $S\propto E$-principal ideal of $A\propto E$.
\end{lem}

\begin{proof}
Assume that $F$ is $S$-cyclic. Then there exist $s\in S$ and $e\in F$ such that \begin{center}
	$sF\subseteq Ae\subseteq F$,\end{center}
 and so \begin{center} $0\propto sF\subseteq 0\propto Ae\subseteq 0\propto F$,\end{center}
 therefore,\begin{center} $(s,0)0\propto F\subseteq A\propto E(0,e)\subseteq 0\propto F$.\end{center}
Hence, $0\propto F$ is an $S\propto E$-principal ideal of $A\propto E$. Conversely, assume that $0\propto F$ is an $S\propto E$-principal ideal of $A\propto E$. Then there exists $(s,e)\in S\propto E$ such that
\begin{center} $(s,e)0\propto F\subseteq A\propto E(a,e')\subseteq 0\propto F$, for some $(a,e')\in A\propto E$,\end{center}
 and so \begin{center} $0\propto sF\subseteq Aa\propto (Ae'+aE)\subseteq 0\propto F$\end{center}
 therefore, \begin{center} $0\propto sF\subseteq 0\propto Ae'\subseteq 0\propto F$.\end{center} Hence, $sF\subseteq Ae'\subseteq F.$ Finally, $F$ is $S$-cyclic, as desired. \qed
\end{proof}

\bigskip

{\parindent0pt {\bf Proof of Theorem \ref{thm2.2}.\ }}
Assume that $R$ is an $S$-B\'ezout ring. Let $F$ be a finitely generated submodule of $E$. Then $0\propto F$ is a finitely generated ideal of $R$ by \cite[Theorem 2(1)]{MMA} and so is $S$-principal. Therefore, $F$ is $S$-cyclic by Lemma \ref{lem2.3}. On the other hand, consider a finitely generated ideal $I$ of $A$. From \cite[Theorem 7(1)]{MMA}, $I\propto IE$ is a finitely generated ideal of $R$. Consequently, there exist $(s,e)\in S$ and $(a,e')\in I\propto IE$ such that

\begin{center}
$(s,e)I\propto IE\subseteq A\propto E(a,e')\subseteq I\propto IE$,\end{center}
 so, \begin{center} $sI\subseteq Aa\subseteq I$.\end{center}
Hence, $I$ is $S_0$-principal, making $A$, an $S_0$-B\'ezout ring. Conversely, assume that every ideal of $R$ is homogeneous (recall from \cite{AW}, that homogeneous ideals of $R$ have the form $I\propto F$ for some ideal $I$ of $A$ and some submodule $F$ of $E$ with $IE\subseteq F$). Let $I\propto F$ be a finitely generated ideal of $R$. By \cite[Theorem 9(1)]{MMA}, $I$ is a finitely generated ideal of $A$ and $F$ is a finitely generated submodule of $E$. So, there exist $(s,s')\in S_0^{2}$, $a\in I$ and $e\in F$ such that
\begin{center}
 $sI\subseteq Aa\subseteq I$,\end{center}
  and \begin{center} $s' F\subseteq Ae\subseteq F$.\end{center}

 Set $t=ss'\in S$, then

 \begin{center}
  $tI\subseteq Aa\subseteq I$\end{center} and \begin{center}$tF\subseteq Ae\subseteq Ae+aE\subseteq F,$\end{center}
  and so \begin{center} $tI\propto tF\subseteq Aa\propto (Ae+aE)\subseteq I\propto F$,\end{center}therefore,\begin{center} $(t,0)I\propto F\subseteq A\propto E(a,e)\subseteq I\propto F$.\end{center}
  Hence, $I\propto F$ is $S_0\propto E$-principal. Finally, $R$ is an $S$-B\'ezout ring, as desired. \qed
\bigskip


The following theorem investigates the
transfer of $S$-B\'ezout ring property in various context of trivial ring extensions.\\
\begin{thm}\label{thm124} Let $A$ be a ring, $E$ be an $A-$module, $S_0$ be a multiplicatively closed subset of $A$, $R:= A\propto E$ be the trivial ring extension of $A$ by $E$ and $S:= S_0\propto E$ be a multiplicatively closed subset of $R$. Then the following statements hold:
\begin{enumerate}
\item If $R$ is an $S$-B\'ezout ring, then $A$ is an $S_0$-B\'ezout ring.
\item Assume that $A$ is an integral domain which is not a field, $K =qf(A)$, and $R:=A
\propto K$ be the trivial ring extension of $A$ by $K$. Then $R$ is an $S$-B\'ezout ring if and
only if $A$  is an $S_0$-B\'ezout domain.
\item Assume that $(A,M)$ is a local ring and $E$ is an $A-$module
such that $ME=0$ and $S_0 \nsubseteq U(R) (=R-(M\propto E)$. Then $R$ is an $S$-B\'ezout ring if and only if $A$ is an $S_0$-B\'ezout ring.
\end{enumerate}
\end{thm}


\bigskip
\begin{proof}
$(1)$ Assume that $R$ is an $S$-B\'ezout ring and let $I$ be a finitely
generated proper ideal of $A$ generated by $(a_i)_{i=1, \ldots ,n}$. Then the finitely generated ideal $J$ of $R$ generated by
$(a_i,0)_{i=1, \ldots ,n}$ is $S$-principal, that is there exist
$(s_0,f) \in S$ and $(a,e) \in R$ such that $$(s_0,f)J \subseteq R(a,e) \subseteq J$$ and so
$$s_0 I \subseteq Aa \subseteq I.$$
Therefore, $I$ is an $S_0$-principal and hence $A$ is an
$S_0$-B\'ezout ring.\\

$(2)$ If $R$ is an $S$-B\'ezout ring, then $A$ is an
$S_0$-B\'ezout domain by assertion $(1)$ above. Conversely, let $J$ be a finitely generated proper ideal of $R$.
Set $I :=\{a \in A / (a,e) \in J$ for some $e \in K\}$. Two cases are then possible: \\
Case 1. $I =0$.\\
Necessarily, $J =0\propto (1/b)L$ for some $b\not=0 \in A$ and some finitely
generated proper ideal $L$ of $A$. Therefore, $L$ is an $S_0$-principal since $A$ is an $S_0$-B\'ezout ring and so there exist $s_0 \in S_0$ and $a\in A$ such
that $$s_0 L \subseteq Aa \subseteq L$$ and we have $$(s_0,0)J \subseteq R(0,a/b) \subseteq J.$$
Hence, $J$ is an $S$-principal ideal of $R$, as desired.\\
Case 2. $I \not= 0$.\\
 Let $(a,e) \in J$ such that $a\not= 0$. Then $(a,e)(0 \propto K) =0 \propto K \subseteq J$;
equivalently, $J  =I \propto IK =I \propto K$, where $I$ is a finitely generated ideal of an $S_0$-B\'ezout ring $A$. Hence, there exist $s_0 \in S_0$ and $a \in A$ such that
$$s_0 I \subseteq Aa \subseteq I,$$ and so we have $$(s_0,0) J \subseteq R(a,0) \subseteq J.$$
Hence, $J$ is an $S$-principal ideal, as desired.\\
$(3)$ If $R$ is an $S$-B\'ezout ring, then A is an $S_0$-B\'ezout ring by assertion $(1)$ above.
Conversely, assume that $A$ is an $S_0$-B\'ezout ring and $S_0 \nsubseteq U(R) (=R-(M\propto E)$, and let $J$ be a finitely generated ideal of $R$.
Set $I :=\{a \in A$ such that $(a,e) \in J$ for some $e \in E\}$ which is a finitely generated ideal of $A$. Since $A$ is an
$S_0$-B\'ezout ring, then there exist $s_0 \in S_0$ and $a \in A$
such that $$s_0 I \subseteq Aa \subseteq I.$$
By multiplying the above equation by an element $s \in S_0 \cap M$, we may assume that $s_0 \in M$ (since we have
$ss_0 I \subseteq Asa \subseteq sI \subseteq I$). On the other hand, there exists $e \in E$ such that $(a,e) \in J$ since $a \in I$. Hence,
$(as_0,0) =(a,e)(s_0,0) \in J$ and so we have
$$(s_0s_0,0)J \subseteq R(s_0,0)(a,e) =R(s_0a,0)  \subseteq J.$$
Therefore, $J$ is an $S$-principal ideal of $R$, which completes the proof.  $\Box$
\end{proof}

\bigskip
 Next, we use the transfer of the $S$-B\'ezout property in the trivial ring extension to provide some new original examples of $S$-B\'ezout rings that are not B\'ezout.

 \begin{exam}\label{exam3.13}
 Let $(A,P)$ be a local B\'ezout domain, $E=A/P$ be an $A$-module such that $Dim_{\frac{A}{P}}(E)\neq 1$ and let $S_0$ be a multiplicatively closed subset of $A$ such that $S_0\cap P\neq \emptyset$. Consider the trivial ring extension of $A$ by $E$ $R:= A\propto E$ and $S:= S_0\propto E$ a multiplicatively closed subset of $R$. Then:
 \begin{enumerate}
 \item $R$ is an $S$-B\'ezout ring.
 \item $R$ is not a B\'ezout ring.
 \end{enumerate}
 \end{exam}
\begin{proof}
$(1)$ Observe that $PE=0$ and $S_0\not\subseteq U(R)$. By assertion $(3)$ of Theorem \ref{thm124}, $R$ is an $S$-B\'ezout ring (as $A$ is an $S_0$-B\'ezout ring).\\
$(2)$ Assume by the way of contradiction that $R$ is a B\'ezout ring. Since $Dim_{\frac{A}{P}}(E)\neq 1$, then $Dim_{\frac{A}{P}}(E)\geq 2$. Pick two elements $h,k\in E$ such that \{h,k\} is $A/P$-linearly independent set and consider the finitely generated ideal $J:=R(0,h)+R(0,k)$ of $R$. One can easily check $J$ is not principal, which is a contradiction. Therefore, $R$ is not a B\'ezout ring.\qed\\
\end{proof}
We denote by $\overline{H}$, the closure integral of integral domain $H$ in its quotient field. Recall from \cite{AZf} that an integral domain $H$ is almost B\'ezout if and only if the integral closure $\overline{H}$ of $H$ is a Pr\"ufer domain with torsion class
group and $H \subseteq \overline{H}$ is a root extension. Now, we show how one can build an example of an $S$-B\'ezout almost B\'ezout ring which is not B\'ezout via trivial ring extension of an integral domain by a vector space over its quotient field, as shown below.

\begin{exam}
Let $F$ be a field of characteristic $p > 0$ and let $F \subseteq L$ be a purely inseparable field
extension. Consider the integral domain $H = F + X L[X],$ $K:=qf(A)$, $R:=H\propto K$ the trivial ring extension of $H$ by $K$ and $S:=S_0\propto E$ be a multiplicative closed subset of $R$, where $S_0:=H-\{0\}$ is a multiplicative closed subset of $H$. Then:
\begin{enumerate}
\item $R$ is an $S$-B\'ezout ring.

\item $R$ is an almost B\'ezout ring.

\item $R$ is not a B\'ezout ring.

\end{enumerate}
\end{exam}

\begin{proof}

$(1)$ By assertion $(2)$ of Theorem \ref{thm124}, $R$ is an $S$-B\'ezout ring, as $H$ is an $S_0$-B\'ezout ring.\\
$(2)$ Note that $\overline{H}=L[X]$ is a Pr\"ufer domain with torsion class group and for each $Q$ in $L[X]$ there exists $n \geq 0$ such that $Q^{p^n} \in H.$ Therefore, $H \subset \overline{H}$ is a root
extension. From \cite[Corollary 4.8(1)]{AZf}, $H$ is an almost B\'ezout domain and so is $R$ \cite[Theorem 3.1(2)]{MaMiMo}.\\
$(3)$ $R$ is not a B\'ezout ring since $H$ is a homomorphic image of $R$ which is not B\'ezout (as $H$ is not integrally closed).\qed\\

\end{proof}

Theorem \ref{thm124} provides a new example of $S$-B\'ezout ring which is not B\'ezout, as shown below.
\begin{exam}\label{exam3.13}
 Let $A$ be a non-B\'ezout integral domain which is not a field, $K=qf(A)$, and $R:=A
\propto K$ be the trivial ring extension of $A$ by $K$. Set $S_0 =A-\{0\}$ and $S :=S_0 \propto E$. Then $R$ is an non-B\'ezout $S$-B\'ezout ring by the assertion $(2)$  Theorem \ref{thm124} (as $A$ is a non-B\'ezout $S_0$-B\'ezout ring).
 \end{exam}

Now, we turn our attention to the transfer of the $S$-B\'ezout ring property
to amalgamation of rings $R :=A\bowtie^{f}J$. It is worthwhile observing that if $S'$ is a  multiplicative closed subset of $R$, then $S_0 = \{a \in A / (a,f(a)+j)\in S'$ for some $ j\in J\}$ is  multiplicative closed subset of $A$. Conversely, if $S_0$ is a multiplicative set of $A$, then $S':=S_{0} \bowtie^{f} 0$ and $S_{0} \bowtie^{f} J$ are multiplicative sets of $R$.\\
\begin{thm}\label{th2.9}
Let $A$ and $B$ be two rings, $J$ an ideal of $B$
and let $f:A\lo B$ be a ring homomorphism. Then:
\begin{enumerate}
\item   If $A\bowtie^fJ$ is an $S'$-B\'ezout ring, then $A$ is an $S_0$-B\'ezout ring.
\item Assume that $f(S_0)\cap J\neq \emptyset$. Then $A\bowtie^fJ$ is an $S'$-B\'ezout ring if and only if $A$ is an $S_0$-B\'ezout ring.
\item  Assume that $f(S_0)\cap J=\emptyset$, $J:=ann(f(S_0))$ and every proper ideal of $A\bowtie^fJ$ is homogenous. Then $A\bowtie^fJ$ is an $S'$-B\'ezout ring if and only if $A$ is an $S_0$-B\'ezout ring.
\end{enumerate}
\end{thm}

\begin{proof}
$(1)$  Assume $R :=A\bowtie^fJ$ is an $S'$-B\'ezout ring and let $I$ be a finitely
generated proper ideal of $A$ generated by $(a_i)_{i=1, \ldots ,n}$. Then the finitely generated ideal $J$ of $R$ generated by
$(a_i,f(a_i))_{i=1, \ldots ,n}$ is $S$-principal, that is there exist
$(s_0,f(s_0)+j) \in S$ and $(a,f(a)+j') \in R$ such that $$(s_0,f(s_0)+j)J \subseteq R(a,f(a)+j') \subseteq J$$ and so
$$s_0 I \subseteq Aa \subseteq I.$$
Therefore, $I$ is $S_0$-principal and hence $A$ is an
$S_0$-B\'ezout ring.\\
$(2)$ Assume that $f(S_0)\cap J\neq \emptyset$. If $A\bowtie^fJ$ is an $S'$-B\'ezout ring, then by assertion $(1)$ above, $A$ is an $S_0$-B\'ezout ring. Conversely, assume that $A$ is $S_0$-B\'ezout. Let $L$ be a finitely generated ideal of $A\bowtie^fJ$. Consider the ideal $I:=\{a\in A / (a,f(a)+j)\in L$ for some $j\in J\}$ of $A$. Since $L$ is a finitely generated proper ideal of $A\bowtie^fJ$, then it is easy to see that $I$ is a finitely generated proper ideal of $A$. Using the fact that $A$ is an $S_0$-B\'ezout ring, then $I$ is an $S_0$-principal ideal of $A$. So, there exist $s_0\in S_0$ and $a\in A$ such that $s_0I\subseteq aA\subseteq I$. From assumption, there is $t\in S_0$ such that $f(t)\in J$. Observe that $ts_0I\subseteq atA\subseteq tI\subseteq I$. Therefore, $(ts_0,0)L\subseteq ts_0I\times 0\subseteq atA\times 0\subseteq (a,f(a))tA\times 0\subseteq (a,f(a))A\bowtie^fJ\subseteq L$, since $f(t)\in J$ and $a\in I$. Hence, $L$ is an $S'$-principal ideal of $A\bowtie^fJ$, making $A\bowtie^fJ$, an $S'$-B\'ezout ring.\\
$(3)$ Assume that $f(S_0)\cap J=\emptyset$, $J:=ann(f(S_0))$ and every proper ideal of $A\bowtie^fJ$ is homogenous. If $A\bowtie^fJ$ is an $S$-B\'ezout ring, then by assertion $(1)$ above, $A$ is an $S_0$-B\'ezout ring. Conversely, assume that $A$ is an $S_0$-B\'ezout ring. Let $L$ be a finitely generated proper ideal of $A\bowtie^fJ$. From assumption, $L:=I\bowtie^fJ$ for some finitely generated proper ideal $I$ of $A$. Since $A$ is an $S_0$-B\'ezout ring, then there exist $s_0\in S_0$ and $a\in A$ such that $s_0I\subseteq aA\subseteq I$. We claim that $(s_0,f(s_0))L\subseteq (a,f(a))A\bowtie^fJ\subseteq L$. Indeed, let $(i,f(i)+k)\in L$. Then $s_0i=ar$ for some $r\in A$. So, $(s_0,f(s_0))(i,f(i)+k)=(s_0i,f(s_0i)+kf(s_0))=(s_0i,f(s_0i))=(ar,f(ar))=(a,f(a))(r,f(r))\in (a,f(a))A\bowtie^fJ\subseteq L$. Hence, it follows that $I\bowtie^fJ$ is an $S'$-principal ideal of $A\bowtie^fJ$. Finally, $A\bowtie^fJ$ is an $S'$-B\'ezout ring, as desired.\qed\\
\end{proof}

Next, let $I$ be a \emph{proper} ideal of $A$.  The (amalgamated) duplication of $A$ along $I$ is a special amalgamation given by $$A\bowtie I:=A\bowtie^{id_{A}} I=\big\{(a,a+i)/ a\in A, i\in I\big\}.$$ The following corollary is an immediate consequence of Theorem \ref{th2.9} on the transfer of the $S$-B\'ezout property into duplications.
\begin{cor}
Let $A$ be a ring and $I$ be an ideal of $A$. Then:
\begin{enumerate}
\item If $A\bowtie I$ is an $S'$-B\'ezout ring, then $A$ is an $S_0$-B\'ezout ring.
\item Assume that $S_0\cap I\neq \emptyset$. Then $A\bowtie I$ is an $S'$-B\'ezout ring if and only if so is $A$ is an $S_0$-B\'ezout ring.
\item Assume that $S_0\cap I=\emptyset$, $I:=Ann(f(S_0))$ and every proper ideal of $A\bowtie I$ is homogenous. Then $A\bowtie I$ is $S'$-B\'ezout if and only if $A$ is an $S_0$-B\'ezout ring.
\end{enumerate}

\end{cor}

Theorem \ref{th2.9} enriches the current literature with a new original class of $S$-B\'ezout rings which are not B\'ezout. The following examples show how to construct such rings.

\begin{exam}
Let $A$ be any $S$-B\'ezout ring which is not B\'ezout (for instance take $A:=B\propto E$ be the trivial ring extension of $B$ by $E$, $B:=\Z$ be the ring of integers, $E:=(\Z/2\Z)^\infty$ be a $\Z/2\Z$-vector space, see example \ref{ex}). Consider the surjective ring homomorphism $f:A\rightarrow B$ and $S_0:=\{(2^n,0)/n\in \N\}$ be a multiplicative set of $A$ and $J:=2\Z$ be an ideal of $B$. Then:
\begin{enumerate}
\item $A\bowtie^fJ$ is an $S'$-B\'ezout ring.
\item $A\bowtie^fJ$ is not a B\'ezout ring.
\end{enumerate}
\end{exam}
\begin{proof}
$(1)$ First, observe that $f(S_0)\cap J\neq \emptyset$. Since $A$ is an $S_0$-B\'ezout ring (by Example \ref{ex}), then by assertion $(2)$ of Theorem \ref{th2.9}, it follows that $A\bowtie^fJ$ is an $S'$-B\'ezout ring.\\
$(2)$ We claim that $A\bowtie^fJ$ is not a B\'ezout ring. Indeed, $A$ which is a homomorphic image of $A\bowtie^fJ$ is not B\'ezout and the fact that the B\'ezout property is stable under factor ring, it follows that $A\bowtie^fJ$ is not a B\'ezout ring.\qed\\
\end{proof}

\begin{exam}
Let $A:=A_1\times A_2$ be a ring (for instance $A:=\Z\times K[X,Y]$ where $\Z$ is the ring of integers and $K[X,Y]$ the polynomial ring in two variables $X,Y$ over a field $K$). Consider the multiplicative sets $S_1:=\{3^n/n\in \N\}$ and $S_2:=K[X,Y]-\{0\}$ of $\Z$ and $K[X,Y]$, respectively, the surjective ring homomorphism $f:\Z\times K[X,Y]\rightarrow \Z$ defined by $f((a,P))=a$ and the ideal $J:=3\Z$ of $\Z$. Let $S=S_1\times S_2$ be a multiplicative set of $A$ and $S':=S\bowtie^f 0$ be a multiplicative set of $A\bowtie^fJ$. Then:
\begin{enumerate}
\item $A\bowtie^fJ$ is an $S'$-B\'ezout ring.
\item $A\bowtie^fJ$ is not a B\'ezout ring.
\end{enumerate}
\end{exam}

\begin{proof}
$(1)$ First note that $A:=\Z\times K[X,Y]$ is an $S_1\times S_2$-B\'ezout ring (by Proposition \ref{prop2.6}, since $\Z$ is an $S_1$-B\'ezout ring and $K[X,Y]$ is an $S_2$-B\'ezout ring). On the other hand, $f(S)\cap J=S_1\cap J\neq \emptyset$. Hence, by assertion $(2)$ of Theorem \ref{th2.9}, $A\bowtie^fJ$ is an $S'$-B\'ezout ring.\\
$(2)$ $A\bowtie^fJ$ is not a B\'ezout ring, since $A$ is not a B\'ezout ring (as $K[X,Y]$ is not a B\'ezout domain).\qed

\end{proof}

The next example shows that the class of $S$-B\'ezout rings and the class of almost B\'ezout rings are distinct in general and we use Theorem \ref{th2.9} to construct a new original class of $S$-B\'ezout rings which are not almost B\'ezout.

\begin{exam}
Let $A$ be any integral domain which is not a field, $f=id_A$ be the identity ring homomorphism and let
$J:=m$ be a maximal ideal of $A$. Consider the multiplicative closed subset $S_0:=A-\{0\}$ of $A$. Then:
\begin{enumerate}
\item $A\bowtie^fJ$ is an $S'$-B\'ezout ring.
\item $A\bowtie^fJ$ is not an almost B\'ezout ring.
\end{enumerate}
\end{exam}
\begin{proof}
$(1)$ Using a similar argument as in Example \ref{0}, it follows that $A$ is  an $S'$-B\'ezout ring. On the other hand $f(S_0)=f(A-\{0\})\cap J=A-\{0\}\cap m \neq \emptyset$. Hence, by assertion $(2)$ of Theorem \ref{th2.9}, it follows that $A\bowtie^fJ$ is an $S'$-B\'ezout ring.\\
$(2)$ From \cite[Corollary 2.5]{NaMo}, $A\bowtie^fJ$ is not an almost B\'ezout ring, as $J\neq 0$.\qed\\
\end{proof}

Now, we introduce the concept of nonnil $S$-B\'ezout ring.

\begin{defn}
Let $A\in \mathcal{H}$ be a ring and $S\subseteq A$ a multiplicative set. $A$ is called a nonnil $S$-B\'ezout ring if for every nonnil finitely generated ideal $I$ of $A$, we have $I$ is an $S$-principal ideal of $A$.
\end{defn}
 Next, we establish a characterization of nonnil $S$-B\'ezout property. For every multiplicative set $S$ of ring $A$, set $S'=\frac{S}{Nil(A)}=\{s+Nil(A)\mid s\in S \}$. It is easy to see that $S'$ is a multiplicative set of $\frac{A}{Nil(A)}$.
\begin{thm}\label{thmo}
 Let $A\in \mathcal{H}$ be a ring and $S\subseteq A$ a multiplicative set. Then $A$ is a nonnil $S$-B\'ezout ring if and only if $\frac{A}{Nil(A)}$ is an $S'$-B\'ezout domain.
\end{thm}
\proof Assume that $A$ is a nonnil $S$-B\'ezout ring and let
$\frac{I}{Nil(A)}$ be finitely generated ideal of
$\frac{A}{Nil(A)}$. We aim to show that $\frac{I}{Nil(A)}$ is an $S'$-principal ideal of $\frac{A}{Nil(A)}$. We may assume that $\frac{I}{Nil(A)}\neq 0$. Hence, $I$ is a nonnil ideal of $A$. By \cite[Lemma 2.4]{aybad}, $I$ is a finitely generated ideal of the nonnil $S$-B\'ezout ring $A$ since $\frac{I}{Nil(A)}$ be finitely generated ideal of
$\frac{A}{Nil(A)}$. So, $I$ is an $S$-principal ideal of $A$. Therefore, there exists $s\in S$ such that $sI\subseteq aA\subseteq I$ for some $a\in A$. It follows that there exist $s'\in S'$ and $\bar{a}\in \frac{A}{Nil(A)} $ such that $s'\frac{I}{Nil(A)}\subseteq \bar{a}\frac{A}{Nil(A)}\subseteq
\frac{I}{Nil(A)}$ with $s'=s+Nil(A)$ for some $s\in S$. So,
$\bar{a}\frac{A}{Nil(A)} = (ab + Nil(A))$ for some
$b\in A$. Hence, $\frac{I}{Nil(A)}$ is an $S'$-principal ideal of $\frac{A}{Nil(A)}$, making $\frac{A}{Nil(A)}$, an $S'$-B\'ezout domain. Conversely, assume that $\frac{A}{Nil(A)}$ is an $S'$-B\'ezout domain and let $I$ be a nonnil ideal finitely generated of $A$. By \cite[Lemma 2.4]{aybad}, $\frac{I}{Nil(A)}$ is a nonzero finitely generated ideal of the $S'$-B\'ezout ring $\frac{A}{Nil(A)}$. So, $\frac{I}{Nil(A)}$ is an $S'$-principal ideal of $\frac{A}{Nil(A)}$. Therefore, there exist $s'\in S'$ and $\bar{a}\in \frac{A}{Nil(A)} $ such that
$s'\frac{I}{Nil(A)}\subseteq \bar{a}\frac{A}{Nil(A)}\subseteq
\frac{I}{Nil(A)}$ with $s'=s+Nil(A)$ for some $s\in S$. So,
$\bar{a}\frac{A}{Nil(A)} = (ab + Nil(A))$ for some
$b\in A$. We have $\bar{a}\frac{A}{Nil(A)}\cong \frac{J}{Nil(A)}$. Our aim is to show that $J$ is a principal ideal of $A$ generated by $a$. Consider a nonnilpotent element $x$ of $J$. Then $x + Nil(A) = ab + Nil(A)$ in
$\frac{A}{Nil(A)}$ for some $b$ in $A$. Therefore, there exists $w
\in Nil(A)$ such that $x + w = ab$ in $A$.
Since $x$ is nonnilpotent, $w=xk$ for some $k\in Nil(A)$. So, $x+w =
x + xk = x(1 + k) =ab$. Using the fact that
$k\in Nil(A)$ and so $1 + k\in U(A)$, it follows that $x \in aA$ and so $J$ is a principal ideal of $A$ generated by $a$. One can
easily check that $sI\subseteq aA=J\subseteq I$ and so $I$ is an
$S$-principal ideal of $A$. Consequently, $A$ is a nonnil $S$-B\'ezout ring, as desired.\qed\\

As a consequence of Theorem \ref{thmo}, we establish the following characterization of nonnil
$S$-B\'ezout rings.
\begin{cor}
Let $A\in \mathcal{H}$ be a ring and $S\subseteq A$ be a
multiplicative set. Then $A$ is a nonnil $S$-B\'ezout ring if
and only if $\phi(A)$ is a nonnil $\phi(S)$-B\'ezout ring.
\end{cor}
\proof Assume that $A$ is a nonnil $S$-B\'ezout ring. By Theorem \ref{thmo}, $\frac{A}{Nil(A)}$ is an $S'$-B\'ezout domain with $S'=\frac{S}{Nil(A)}$. From \cite[Lemma 2.1]{ayb6}, $\frac{A}{Nil(A)}\cong\frac{\phi(A)}{Nil(\phi(A))}$ and therefore $\frac{\phi(A)}{Nil(\phi(A))}$ is a $\frac{\phi(S)}{Nil(\phi(S))}$-B\'ezout domain. Hence, $\phi(A)$ is a nonnil $\phi(S)$-B\'ezout ring by Theorem \ref{thmo}. The converse holds using similar argument as previously.\qed\\

The following example illustrates Theorem \ref{thmo} by generating a new original class of nonnil $S$-B\'ezout ring which is not B\'ezout.
\begin{exam}
Let $A$ be any non-B\'ezout $S$-B\'ezout domain which is not a field with
quotient field $K$ and let $S\subseteq A$ be a multiplicative set.
Let $K$ be the quotient field of $A$. Then $A\propto K$ is a nonnil
$S$-B\'ezout ring which is not B\'ezout.
\end{exam}
\proof Set $R=A\propto K$. First, $Nil(R) = 0\propto K$ is a divided
prime ideal of $R$. Indeed, let $(0, e) \in Nil(R)$ and $(a, f) \in
R \setminus Nil(R)$. Then $(0, e) = (a, f)(0, \frac{e}{a})$ and thus
$R \in \mathcal{H}$. Since $\frac{R}{Nil(R)}$ is ring-isomorphic to
$A$, $R$ is a nonnil $S$-B\'ezout ring by Theorem \ref{thmo}. Furthermore, $R$ is not a B\'ezout ring.\qed\\

The following theorem establishes a characterization of nonnil $S$-B\'ezout rings in a special setting of pullback.
\begin{thm}
Let $A \in \mathcal{H}$ and $S\subseteq A$ a multiplicative set. Then $A$ is a nonnil $S$-B\'ezout ring if and only if $\phi(A)\cong R$, where $R$ is obtained from the following pullback diagram:
\[
 \xymatrix{
R \ar[d]  \ar[r] & B\ar[d] \\
T \ar[r] & T/M
}
\]
where $T$ is a zero-dimensional quasilocal ring with maximal ideal $M$, $B:=R/M$
is a $S_1$-B\'ezout subring of $T/M$ with $S_1=\alpha(\phi(S))/M$ such that $\alpha$ is the ring isomorphism from $\phi(A)$ to $R$, the vertical arrows are the usual inclusion maps, and
the horizontal arrows are the usual surjective maps.
\end{thm}
\proof Assume that $\phi(A)\cong R$
obtained from the given diagram. Then $R \in \mathcal{H}$ and
$Nil(R) = Z(R) = M$. Since $R/M$ is an $S_1$-B\'ezout domain,
$R$ is a nonnil $S_2$-B\'ezout ring by Theorem \ref{thmo},
where $S_2=\alpha(\phi(S))$, and so $\phi(A)$ is a nonnil $\phi(S)$-B\'ezout ring. Hence, $A$ is a nonnil
$S$-B\'ezout ring. Conversely, assume that $A$ is a nonnil
$S$-B\'ezout ring. Set $T = A_{Nil(A)}$,
$M = Nil(A_{Nil(A)})$, and $R = \phi(A)$ yields the desired pullback diagram. \qed\\

\begin{section}{Conflict of Interest}

On behalf of all authors, the corresponding author states that there is no conflict of interest.\\
\end{section}

\end{document}